\documentclass[11pt,conference]{IEEEtran}

\usepackage{amsmath,amsfonts,amssymb,enumerate}
\usepackage{bbm}
\usepackage{gensymb, array}
\usepackage{latexsym}
\usepackage{graphicx}
\usepackage{multirow,rotating}
\usepackage{algpseudocode,algorithm}
\usepackage{epstopdf}
\usepackage[margin=0.5cm]{subcaption}
\IEEEoverridecommandlockouts

\newtheorem{theorem}{\bf Theorem}[section]

\newenvironment{proof}{\noindent{\em Proof:}}{\quad \hfill$\Box$\vspace{2ex}}

\newcommand{\bn}{{\bf n}}
\newcommand{\bx}{{\bf x}}
\newcommand{\LL}{L_\omega^2(\mathbb{R}^d)}
\newcommand{\argmin}{\operatornamewithlimits{argmin}}

\title{Sparse generalized Fourier series via collocation-based optimization \thanks{Approved for public release by USAF 88th 
ABW on 29 Sep 2014.  Case Number: 88ABW-2014-4587.  Presented at 2014 Applied Imagery and Pattern Recognition Workshop. }}

\author{\IEEEauthorblockN{Ashley Prater}\IEEEauthorblockA{Air Force Research Laboratory \\{\tt ashley.prater.3@us.af.mil}}}
\date{\today}

\begin{document}
\maketitle

\begin{abstract}
Generalized Fourier series with orthogonal polynomial bases have useful applications in several fields, including differential equations, pattern recognition, and image and signal processing.  However, computing the generalized Fourier series can be a challenging problem, even for relatively well behaved functions.  In this paper, a method for approximating a sparse collection of Fourier-like coefficients is presented that uses a collocation technique combined with an optimization problem inspired by recent results in compressed sensing research.  The discussion includes approximation error rates and numerical examples to illustrate the effectiveness of the method.  One example displays the accuracy of the generalized Fourier series approximation for several test functions, while the other is an application of the generalized Fourier series approximation to rotation-invariant pattern recognition in images.

\end{abstract}

\section{Introduction}

This paper discusses an efficient method to approximate the sparse generalized Fourier series of a given function in terms of orthogonal polynomials in several dimensions.  That is, given a function $f:\mathbb{R}^d\to\mathbb{R}$ satisfying certain properties, we seek to compute the Fourier-like coefficients $\{\hat{c}_\bn:\bn\in W\}$ such that 
\begin{equation}\label{eq:sparse generalized FH series}
	f \approx \sum_{\bn\in W} \hat{c}_\bn \pi_\bn,
\end{equation}
where $\{\pi_\bn\}$ is a collection of orthogonal polynomials and $W$ is is a sparse collection of multi-indices.

To compute the sparse collection of Fourier-like coefficients, we propose using a collocation method in an optimization scheme motivated by recent results in the field of compressed sensing.  This method seeks a sparse collection of coefficients such that the finite sum appearing in~\eqref{eq:sparse generalized FH series} nearly interpolates $f$ at carefully chosen nodes.  

The notion of sparsity has been increasingly emphasized in recent literature.  As the amount of useful data collected continues to increase, the extraction of meaningful data is often hindered by the so-called `curse of dimensionality'.  Too alleviate its effects, methods exploiting any natural sparse or hierarchical structures in data have been developed.  Research in sparse grids, particularly when coupled with a hierarchical scheme using Smolyak's algorithm, laid the foundation for accurate computation of high dimensional problems while significantly reducing the computational complexity from $\mathcal{O}(N^d)$ to $\mathcal{O}(N \log^{d-1}N)$, where $d$ is the underlying dimensionality of the problem and $N$ is the number of grid points in each coordinate direction~\cite{bungartz, zenger}.

For linear regression problems, the importance of sparsity has emerged in recent advances in compressed sensing.  Given a sensing matrix $X\in\mathbb{R}^{m\times p}$ with $m\ll p$ and a collection of under sampled linear measurements $y = Xc$, using compressed sensing techniques one may approximate the sparse vector $c$ via the optimization problem
\begin{equation*}
	\tilde{c} = \argmin \|c\|_1 \text{ satisfying } Xc = y.
\end{equation*}
Under certain sparsity and coherence conditions, the recovery will be exact~\cite{baraniuk, candestao, candestao2, donohohuo}.  

On the other hand, sparse generalized Fourier series approximations with orthogonal polynomial bases typically use spectral or pseudospectral approximations on sparse grids~\cite{constantine2012, jiangxu2010, prater, shenreview, shen2010}.  These methods determine the sparse collection of multi-indices a priori, which may lead to unintentional deletion of significant terms in the approximation.  To compute the coefficients in the aforementioned works, detailed hierarchical quadrature methods have been developed to perform the numerical integration required for the desired accuracy while keeping the computational complexity low.  Although effective, these methods can be difficult to implement.

In comparison, the method proposed in this paper has a straightforward implementation.  In addition, a major strength of the proposed method is that the sparse index set does not need to be selected in advance.  Instead, the optimization scheme will select the best sparse selection of multi-indices to approximate the function.

The following notation will be used throughout this work.  Vectors will be denoted by boldface letters.  Let $\mathbb{R}$ denote the real numbers, $\mathbb{N}$ the natural numbers and $\mathbb{N}_0 = \mathbb{N}\cup \{0\}$.  Let $\mathbb{R}^d$ denote the usual $d$-dimensional Euclidean space with obvious analogues $\mathbb{N}^d$ and $\mathbb{N}_0^d$.  Let $n_j$ denote the $j^\text{th}$ entry of the vector $\bn$.  Let $C(\Omega)$ equal the space of all continuous real-valued functions defined on a domain $\Omega$.   The space of $\omega$-weighted square-Lebesgue integrable functions over $\mathbb{R}^d$ is denoted by $\LL$ with the associated weighted inner product $\langle \cdot, \cdot \rangle_\omega$ defined by
$\langle f, g \rangle_{\omega} = \int_{\mathbb{R}^d} f g \; \mathrm{d}\omega$ for all $f,g\in L_\omega^2(\mathbb{R}^d)$.  The $\omega$-weighted inner product naturally induces the weighted norm $\|\cdot\|_\omega$.

The rest of this paper is organized as follows.  In Section~2, the proposed method for approximating the sparse generalized Fourier series of a function is presented, along with an analysis of the approximation error.  The advantages of the proposed method are stated and compared to the difficulties exhibited by existing methods.  Section~3 includes numerical experiments in which the proposed method is implemented.  The first set of experiments approximates the sparse generalized Fourier series of several bivariate test functions.  The second set of experiments applies the sparse approximated Fourier-like coefficients to compute the rotation-invariant Gaussian-Hermite moments, and uses these moments to perform classification of unknown rotated images.  The paper concludes with a summary and directions of possible future work.
%
%
%
%

\section{Sparse Generalized Fourier Series with Orthogonal Polynomial Bases}
The proposed method approximates a given function by a sparse generalized Fourier series with an orthogonal polynomial basis.  After selecting a set of multi-indices and evaluating the orthogonal polynomials at pre-determined nodes, the sparse collection of coefficients is computed using a collocation method with a convex optimization problem.  This section explores the multi-index and node selection, thoroughly explains the collocation model used to approximate the Fourier-like coefficients and discusses the approximation errors.  We begin with preliminary facts about orthogonal polynomials and generalized Fourier series.

\subsection{Full and sparse grid orthogonal polynomial representations}
Let $\Omega\subseteq\mathbb{R}^d$.  Let $\omega:\Omega\to\mathbb{R}$ be a continuous weight function.  That is, suppose $\omega\geq 0$ on $\Omega$ and $\omega$ satisfies 
\[ \int_\Omega \mathrm{d}\omega = 1. \]
Associated with each weight function is a collection of orthogonal polynomials.  A set $\{\pi_\bn : \bn \in \mathbb{N}_0^d\}$ satisfying both of the following conditions is the collection of $\omega$-orthogonal polynomials:
\begin{enumerate}
	\item Each $\pi_\bn$ is a multivariate polynomial.
	\item $\displaystyle \left\langle \pi_\bn, \pi_{\bf m}\right\rangle_\omega := \int_\Omega \pi_\bn \pi_{\bf m} \mathrm{d}\omega = 0$ if $\bn \neq {\bf m}$.
\end{enumerate}
Some common collections of univariate orthogonal polynomials and their associated weight functions are given in Table~\ref{tab:orthogonal polys}~\cite{abramowitz}. 

\begin{table}[bhtp]
\begin{tabular}{l|c|c}
Name & $\omega$	&$\Omega$\\
\hline
Chebyshev (First Kind)		&$1/\sqrt{1-x^2}$	&$[-1,1]$	\\
Chebyshev (Second Kind)		&$\sqrt{1-x^2}$	&$[-1,1]$	\\
Legendre					&$1$				&$[-1,1]$	\\
Laguerre					&$e^{-x}$			&$[0,\infty)$	 \\
Hermite					&$e^{-x^2}$		&$\mathbb{R}$	\\
\end{tabular}
\caption{Some common univariate orthogonal polynomials, their associated weight functions $\omega$, and their domains $\Omega$.}
\label{tab:orthogonal polys}
\end{table}

Let $\{\pi_\bn: \bn\in\mathbb{N}_0^d\}$ be a collection of multivariate orthogonal polynomials relative to the weight $\omega$.  Suppose the orthogonal polynomials have been normalized so that $\|\pi_\bn\|=1$ for each $\bn$.  Since the orthogonal polynomials relative to the weight $\omega$ form a complete orthogonal basis of $L_\omega^2(\mathbb{R}^d)$, any function $f\in L_\omega^2(\mathbb{R}^d)$ can be expressed as 
\begin{equation}\label{eq:FH}
	f = \sum_{\bn \in \mathbb{N}_0^d} c_\bn \pi_\bn,
\end{equation}
where the equality is understood to mean convergence in norm and the Fourier-like coefficients are given exactly by
\begin{equation}\label{eq:fourier coefficients}
	c_\bn = \langle f, \pi_\bn \rangle_\omega.
\end{equation}
The expression in Equation~\eqref{eq:FH} is called the exact generalized Fourier series of $f$.  

The computation of Equation~\eqref{eq:FH} presents several difficulties.  Not only does the full-grid generalized Fourier series involve infinitely many terms, but the Fourier-like coefficients can be difficult or impossible to compute exactly even for seemingly innocuous functions $f$.  Moreover, it is challenging to implement numerical integration techniques achieving the desired accuracy due to the highly oscillatory behavior of the integrands. Therefore methods to estimate the generalized Fourier series~\eqref{eq:FH} must tackle both the series truncation and coefficient computation issues while ensuring that the approximation error for one does not overwhelm the other. 

To tackle the series truncation difficulty, we explore three sets of multi-indices.  Fix a positive integer $N$ and consider the sets of candidate full-grid multi-indices
\[ Y_N^d = \left\{ \bn \in \mathbb{N}_0^d : \bn \leq N \right\}, \]
\[ T_N^d = \left\{ \bn \in \mathbb{N}_0^d : \|\bn\|_1 = N \right\},\]
and the sparse-grid hyperbolic cross-shaped multi-indices
\[ S_N^d = \left\{ \bn \in \mathbb{N}_0^d : \prod_{j=1}^d (n_j + 1)\leq N+1 \right\}. \]
Clearly $|Y_N^d| = \mathcal{O}(N^d)$ and $|S_N^d| = \mathcal{O}(N\log^{d-1}N)$.  It can be shown that $S_N^d \subset Y_N^d$~\cite{bungartz}.  These pre-determined multi-indices are commonly used in spectral and pseudospectral methods~\cite{jiangxu2010, prater, shenreview, shen2010}.

Define the full and sparse-grid truncated generalized Fourier series as
\begin{equation}\label{eq:truncation}
	f_W := \sum_{\bn \in W} c_\bn \pi_\bn
\end{equation}
for $W\in\{Y_N^d, T_N^d, S_N^d \}$.

The error rates for the truncated orthogonal polynomial representations of a function~\eqref{eq:truncation} depend on the order of the approximating polynomials as well as the behavior of the function itself.  
\begin{theorem} \label{th:truncation}
For any $s>0$ and $f\in L_\omega^2(\mathbb{R}^d)$,
\[ \left\| f - f_W\right\|_\omega \leq N^{-s} \|f\|_{\kappa^s}, \]
where $W \in \left\{ Y_N^d, T_N^d, S_N^d\right\}$ and $\|\cdot\|_{\kappa^s}$ is the Korobov norm of order $s$ induced by the inner product
\[	\left\langle f, g \right\rangle_{\kappa^s} = \sum_{\bn\in \mathbb{N}_0^d} \langle f, \pi_\bn \rangle_\omega \langle g,\pi_\bn\rangle_\omega \prod_{j=1}^d (1+n_j)^{2s}. \]
\end{theorem}

The proof of Theorem~\ref{th:truncation} regarding the accuracy of the full and sparse-grid truncated generalized Fourier series appears in~\cite{jiangxu2010, prater} for Hermite polynomial bases.  The proof can easily be generalized to bases of other orthogonal polynomial classes.

\subsection{Fourier-like coefficients}
To tackle the coefficient computation difficulty, an optimization problem with a collocation method will be employed.  To this end, Let $\Lambda \subset \mathbb{R}^d$ be a finite collection of nodes with $|\Lambda|=m$, enumerated so that $\Lambda = \left\{ \bx_1, \bx_2,\ldots,\bx_m\right\}$.  Choose $N$, let $W\in\{Y_N^d, T_N^d, S_N^d\}$ and say $|W| = p$.  Enumerate $W$ so that $W = \left\{\bn_1, \bn_2,\ldots,\bn_p\right\}$.  Define the $j^\text{th}$-row $k^\text{th}$-column entry of the collocation matrix $X \in \mathbb{R}^{m\times p}$ by $X_{j,k} = \pi_k(\bx_j)$, and let the $j^\text{th}$ entry of the vector ${\bf f}\in\mathbb{R}^m$ be  given by ${\bf f}_j = f(\bx_j)$.  

The collocation method seeks a sparse collection of coefficients $\{\tilde{c}_\bn: \bn \in W\}$ that satisfies
\begin{equation}\label{eq:collocation}
	f(\bx_j) = \sum_{\bn \in W} \tilde{c}_\bn \pi_\bn (\bx_j) \quad \forall \bx_j \in \Lambda.
\end{equation}
More succinctly, Equation~\eqref{eq:collocation} can be expressed as ${\bf f}~=~X\tilde{\bf c}$.

Although the Jacobi-like collocation matrix $X$ is full rank, it is not necessarily square.  Indeed, we are interested in the case $m\ll p$. One could relax the equality in~\eqref{eq:collocation} and seek $\tilde{\bf c}$ that minimizes $\|{\bf f} - X \tilde{\bf c}\|_2$.  This is exactly a least squares optimization problem with solution $\tilde{\bf c} = X^\dagger {\bf f},$ where $X^\dagger$ is the pseudoinverse of $X$. However, the least squares solution does not promote sparsity.  Unless one knew a priori the exact sparse set $W$ to use, it is unlikely using least squares would result in a sparse representation of $f$.  

Instead, we propose approximating the collocation coefficients using the Dantzig selector, which is a solution to the optimization problem
\begin{equation}\label{eq:dantzig}
	\begin{cases} \mathrm{minimize} \|{\bf c}\|_1 \\
	\text{subject to} \left\| D^{-1} X^\top (X{\bf c} - {\bf f}) \right\|_{\infty} \leq \delta
	\end{cases},
\end{equation}
where $D$ is the $m\times m$ diagonal matrix normalizing the columns of $X$ and the scalar $\delta>0$ is small.  Let $\hat{\bf c}$ be a solution of the optimization problem~\eqref{eq:dantzig}.  Then a sparse generalized Fourier-series approximation to the function $f$ is given by
\begin{equation}\label{eq:approx final}
	\hat{f} = \sum_{\bn \in W} \hat{\bf c}_\bn \pi_\bn.
\end{equation}

Several methods exist to quickly and accurately solve the optimization problem~\eqref{eq:dantzig}, including an alternating direction method~\cite{lupong} and an iterative method based upon proximity operators~\cite{pratershenDantzig}.

Before performing the error analysis of the approximation~\eqref{eq:approx final}, some notation must first be presented. For a domain $\Omega\subseteq\mathbb{R}^d$, let $C(\Omega)$ denote the class of all real-valued continuous functions on $\Omega$, and let 
\[C^m(\Omega) = \{ f : f^{(\alpha)} \in C,\; \forall |\alpha|\leq m \}.\]  
For any positive integer $m$, define a seminorm\\ $|\cdot|_{X_m(\Omega)}$ on $C^m(\Omega)$ as 
\[|f|_{X_m(\Omega)} = \max \left\{ |f^{(\alpha)}({\bf x})| : {\bf x}\in\Omega, |\alpha|\leq m \right\}.\]
For any function $f:\mathbb{R}^d\to\mathbb{R}$, let the supnorm of $f$ restricted to the domain $\Omega$ be given by $\| f \|_{\infty,\Omega}$.

The error in the sparse orthogonal polynomial interpolation of a function $f$ greatly depends on the highest order of the polynomials used in the summation as well as the smoothness and behaviour of $f$ itself.
\begin{theorem}\label{th:interpolation error}
	If the nodes $\Lambda = \left\{ {\bf x}  : \pi_{M+1}(x_j) = 0 \right\} \subseteq \mathbb{R}^d$ are a rectangular grid of the zeros of the $M+1^\text{th}$ orthgonal univariate polynomial and if $f\in L_\omega^2(\mathbb{R}^d)\cap C^m(\mathbb{R}^d)$, then for any simply connected domain $\Omega\subseteq \mathbb{R}^d$ there exists a constant $c>0$ such that 
\[	\left\| f - \sum_{\bn \in Y_N^d} \tilde{\bf c}_\bn \pi_\bn \right\|_{\infty, \Omega} \leq c 2^{-mM}\left| f \right|_{X^m_M(\Omega)}.\]
\end{theorem}
The proof of Theorem~\ref{th:interpolation error} appears in~\cite{prater} for multivariate Hermite orthogonal polynomials.  The generalization to other classes of orthogonal polynomials is straightforward.

Finally, since the polynomials $\{\pi_\bn\}$ are orthonormal relative to $\omega$, the error in approximating $\hat{f}$ by the interpolated polynomial $\tilde{f} = \sum_{{\bf n}\in W} \tilde{c}_{\bn} \pi_\bn$ is characterized by the difference in their coefficients:
\[	\left\| \hat{f} - \tilde{f} \right\|^2_\omega	= \left\| \sum_{\bn \in W} (\hat{c}_\bn - \tilde{c}_\bn)\pi_\bn\right\|_\omega^2 = \left\| \hat{c}_\bn -\tilde{c}_\bn\right\|_\omega^2,\]
which is guaranteed to be small provided $\tilde{c}$ has a small $\ell_1$ norm.  Indeed, the constrant in~\eqref{eq:dantzig} forces the coefficient error to be small between the recovered polynomial and the interpolated polynomial.
%
%
%
%
%
%
%
%

\section{Numerical Examples}
In the following experiments, the sparse generalized Fourier series approximation is computed for various test functions using the collocation approach presented in Section~2.  In Experiment 1, various node and multi-index selection methods are used for several test functions.  The accuracy of the approximation for each test function is presented.  To illustrate the utility of the proposed method, the collocation method presented above is used to quickly generate the Gausian-Hermite moments of several images in a training set in Experiment 2, which are then used for classification of rotated images in a testing set.  

The Hermite polynomials are used as the orthogonal basis in the sparse representations.  Consider the collection of orthogonal univariate Hermite polynomials with associated weight function $\omega(x,y) = e^{-(x^2+y^2)}$.   The first few univariate Hermite polynomials prior to normalization are $\pi_0(x) = 1, \; \pi_1(x) = 2x, \; \pi_2(x) = 4x^2-2$, and satisfy the three term recurrence
\[	\pi_{n+1}(x) = 2x\pi_n(x) - 2n\pi_{n-1}(x) \]
for $n\in\mathbb{N}$.  Since $\omega$ is separable, the multivariate Hermite polynomials are products of the univariate ones:
\[ \pi_\bn({\bf x}) = \pi_{n_1}(x_1)\pi_{n_2}(x_2)\cdots\pi_{n_d}(x_d).\]
The polynomials are then normalized as 
\[\pi_n \leftarrow \pi_n/\sqrt{n!2^n\sqrt{\pi}}.\]

To find the solution to the optimization problem~\eqref{eq:dantzig}, we use the proximity operator based iterative method recently proposed in~\cite{pratershenDantzig}.  The proximity operator based method is straightforward and uses only two soft thresholding operations at each iteration, requiring $\mathcal{O}(4mp)$ multiplications in each iteration for a $m\times p$ collocation matrix $X$.

All experiments are performed in MATLAB 2014a on a PC with an Intel Core i7-3630QM 2.40 GHz processor and 16GB RAM running Windows 7 Enterprise with machine precision $2.2204e-16$.

{\bf Experiment 1.} 
\begin{figure*}
\centering
	\includegraphics[width=0.3\textwidth]{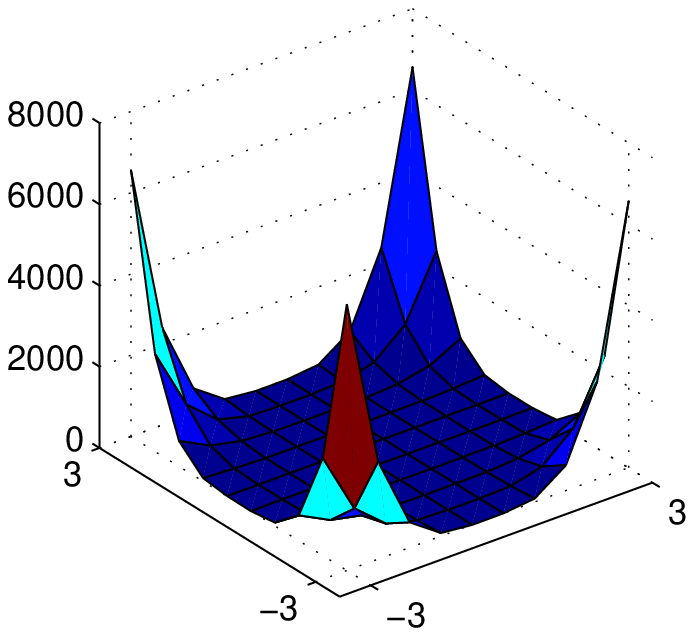}
	\includegraphics[width=0.30\textwidth]{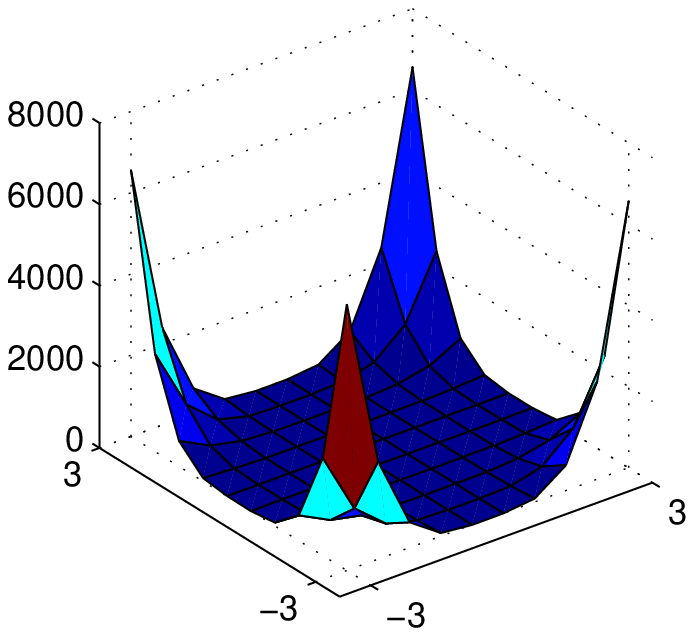}
	\includegraphics[width=.3\textwidth]{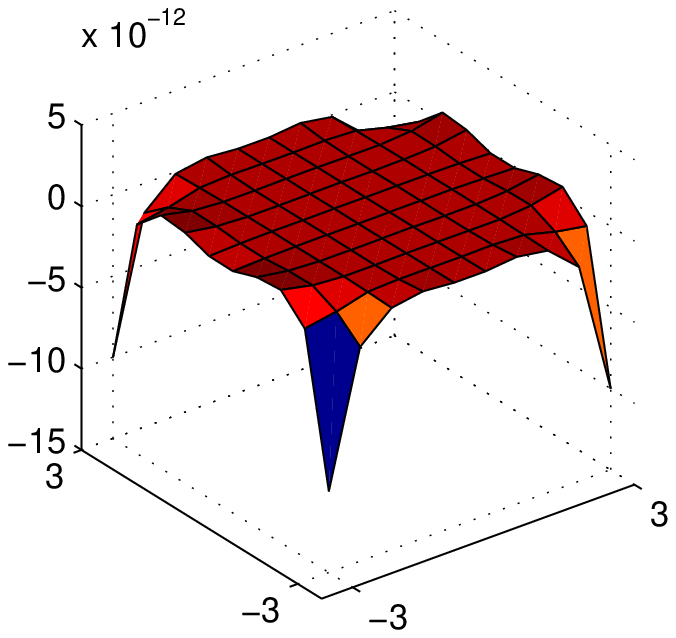}

	\includegraphics[width=0.3\textwidth]{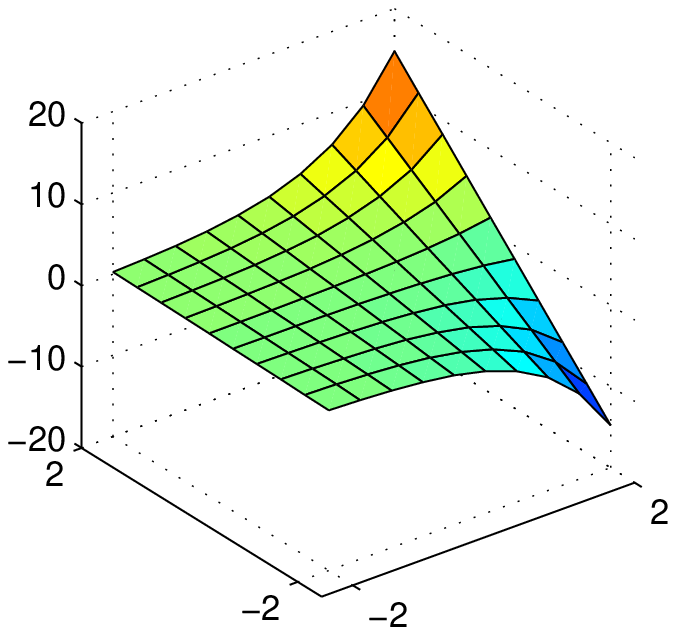}
	\includegraphics[width=0.3\textwidth]{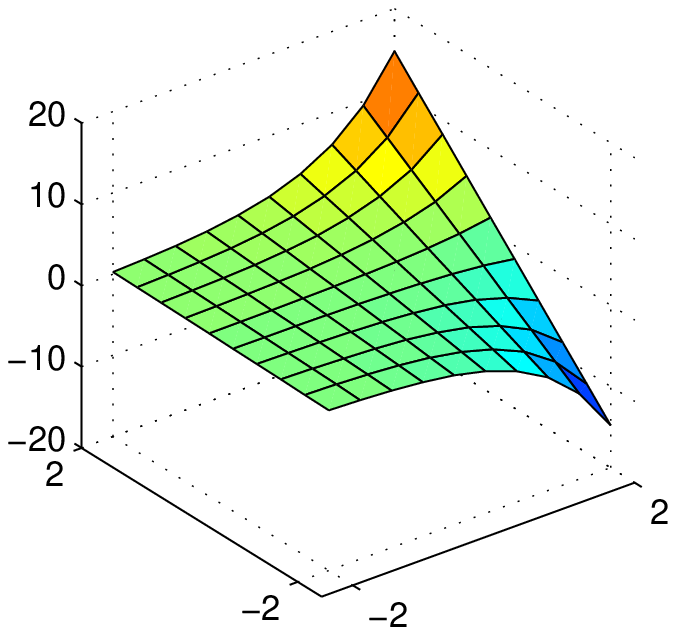}
	\includegraphics[width=0.3\textwidth]{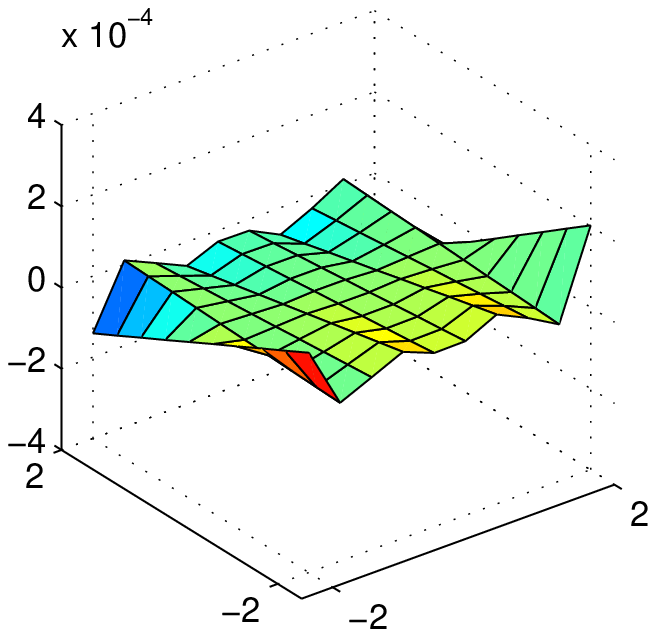}
%
%
\caption{ The exact function (left), the approximated generalized Fourier series (center) and the pointwise approximation error (right) for the functions $f_2$ (top row) and $f_3$ (bottom row) in Example 1.  The approximations for $f_2$ use $Y_5^2$ as the collection of multi-indices and $M=5$.  The approximations for $f_3$ use $Y_{10}^2$ and $M=10$. }
\label{fig:example 1 f2}
\end{figure*}

In this experiment, the sparse generalized Fourier series approximation several given functions are computed.  Suppose $d=2$ and consider the functions
\begin{align*}
	f_1(x,y) &= x^2y^2 \\
	f_2(x,y) &= x^4y^4\\
	f_3(x,y) &= xe^{y/2}.
\end{align*}
It is easy to check that $f_j \in L^2_\omega(\mathbb{R}^2)$ for each $j$.

To create the collocation matrix $X$ for use in the scheme~\eqref{eq:dantzig},  we use the full-grid rectangular, triangular and a sparse-grid hyperbolic cross shaped set of candidate multi-indices, $Y_N^d, T_N^d$ and $S_N^d$, along with nodes formed by a rectangular collection the zeros of the $M$ degree univariate Hermite polynomials. The parameter $N$ determines the highest degree hermite polynomial used in the approximation, and the parameter $M$ determines the number of nodes used as collocation points.

\begin{table*}
	\centering
	\begin{tabular}{c|c|c|cccccccc}
	\multicolumn{1}{c}{}&\multicolumn{1}{c}{}&&\multicolumn{8}{c}{$N$} \\
	\multicolumn{1}{c}{}&\multicolumn{1}{c}{$W$}&$M$&2&3&4&5&6&7&8&9\\
	\hline
	\multirow{9}{*}{\begin{sideways} $\hat{f}_1$\end{sideways}}
	&\multirow{3}{*}{\begin{sideways}$Y_N^2$\end{sideways}} 
&N-1	        &3.1250e-01   &1.8750e-01   &1.5830e+00   &6.9389e-17   &1.3597e-16   &2.3967e-16   &8.7694e-16   &2.0015e-16   \\
&&N  	 	&1.8750e-01   &2.3097e-16   &6.9389e-17   &1.6303e-16   &2.3967e-16   &1.2178e-16   &2.0015e-16   &1.2478e-15   \\
&&N+1   	&2.3097e-16   &6.9389e-17   &1.6303e-16   &2.3967e-16   &1.2178e-16   &2.0015e-16   &1.2478e-15   &5.4918e-16    \\
\cline{2-11}
&\multirow{3}{*}{\begin{sideways}$T_N^2$\end{sideways}}
&N-1		&3.0619e-01   &1.7678e-01   &1.7922e-01   &6.9389e-17   &7.9722e-17   &6.2450e-17   &3.0938e-16   &6.7987e-17\\
&&N 		&1.7678e-01   &3.0619e-01   &6.9389e-17   &1.2432e-16   &6.2450e-17   &9.3095e-17   &6.7987e-17   &7.8505e-17  \\
&&N-1 	&3.0619e-01   &1.0607e+00   &1.2432e-16   &1.6997e-16  & 9.3095e-17   &4.6548e-16   &7.8505e-17   &5.1962e-15 \\
\cline{2-11}
&\multirow{3}{*}{\begin{sideways}$S_N^2$\end{sideways}}
&N-1 	&3.0619e-01   &1.7678e-01   &3.0619e-01   &1.0607e+00   &2.3117e+00   &4.0620e+00   &1.1938e-16   &1.6464e-15 \\
&&N 		&1.7678e-01   &3.0619e-01   &1.0607e+00   &2.3117e+00   &4.0620e+00   &6.3122e+00   &1.6464e-15   &2.5542e-15 \\
&&N+1 	&3.0619e-01   &1.0607e+00   &2.3117e+00   &4.0620e+00   &6.3122e+00   &9.0623e+00   &2.5542e-15   &6.4896e-15\\
\hline\hline
	\multirow{9}{*}{\begin{sideways}$\hat{f}_3$\end{sideways}}
&	\multirow{3}{*}{\begin{sideways}$Y_N^2$\end{sideways}} 
&N-1      	&7.2220e-01   &7.3931e-01   &8.2504e-02   &1.3435e-02   &1.6915e-03   &2.6417e-04   &2.0687e-04   &2.0651e-04   \\
&&N 		&9.5017e-02   &1.6675e-02   &2.1038e-03   &2.5979e-04   &2.0466e-04   &2.0642e-04   &2.0650e-04   &2.0651e-04   \\
&&N+1 	&1.0174e-02   &1.2818e-03   &1.9914e-04   &2.0421e-04   &2.0642e-04   &2.0650e-04   &2.0651e-04   &2.0627e-04   \\
\cline{2-11}
&	\multirow{3}{*}{\begin{sideways}$T_N^2$\end{sideways}} 
&N-1 	&7.1773e-01   &6.6775e-01   &1.6675e-02   &2.1038e-03   &2.5979e-04   &2.0466e-04   &2.0642e-04   &2.0650e-04   \\
&&N 		&5.0888e-02   &1.0174e-02   &1.2818e-03   &1.9914e-04   &2.0421e-04   &2.0642e-04   &2.0650e-04   &2.0651e-04   \\
&&N+1 	&1.4089e-01   &4.8298e-02   &1.5351e-02   &3.8721e-03   &9.1676e-04   &3.1805e-04   &2.1067e-04   &2.0025e-04   \\
\cline{2-11}
&	\multirow{3}{*}{\begin{sideways}$S_N^2$\end{sideways}} 
&N-1		&6.4201e-01   &6.6292e-01   &1.4089e-01   &4.8298e-02   &1.0762e-01   &4.4839e-02   &9.2400e-02   &2.5983e-02   \\
&&N 		&1.1713e-02   &1.4089e-01   &3.2214e-01   &1.0762e-01   &1.7750e-01   &9.2400e-02   &1.6179e-01   &4.6104e-02   \\
&&N+1 	&1.4085e-01   &3.2214e-01   &5.3508e-01   &1.7750e-01   &2.6136e-01   &1.6179e-01   &2.5687e-01   &6.7726e-02   \\
\hline
\end{tabular}
\caption{The $\ell_2$ vector norm error $\|{\bf c} - \hat{\bf c}\|_2$ for the generalized sparse Fourier series expansions $\hat{f}_1$ and $\hat{f}_3$ from Experiment 1 for various combinations of $N$ and $M$.
}
\label{tab: experiment 1}
\end{table*}

Figure~\ref{fig:example 1 f2} displays the sparse generalized Fourier Hermite approximations of the functions $f_2, f_3$ with the coefficients computed as in~\eqref{eq:dantzig}, and Table~\ref{tab: experiment 1} gives the error in computing the coefficients as in~\eqref{eq:dantzig} of the test function $f_1$ for each multi-index scheme with $M \in \{N-1, N, N+1\}$ for $N = 2, 3, \ldots, 9$.  In particular, the exact Fourier Hermite series expansion of $f_1$ has exactly 4 nonzero coefficients located at the multi-indices $I_1 = \{(0,0), (0,2), (2,0), (2,2)\}$.  As seen in Table~\ref{tab: experiment 1}, the locations and values of the nonzero coefficients were approximated well provided $I_1\subset W$, and not well when $I_1\not\subset W$.  Similarly, the exact Fourier Hermite series expansion of $f_2$ has exactly 9 nonzero coefficients located at $I_2 = \{(n_1,n_2): n_j = 0,2,4\}$.  On the other hand, the exact Fourier Hermite series expansion of $f_3$ contains infinitely many nonzero coefficients.  However, as in the Taylor expansion of the exponential function, the coefficients decay quickly and an accurate approximation can be obtain by using only a few terms, as seen in Figure~\ref{fig:example 1 f2}.

From Table~\ref{tab: experiment 1} one can see the drawbacks of using the predetermined sparse multi-index set $S_N^d$.  The parameter $N$ must be much larger to achieve accuracy on par with the full-grid rectangular and triangular index set implementations, even for a function as simple as $f_1$.

{\bf Experiment 2.}
In this experiment, the sparse Fourier-Hermite series approximation is used to perform pattern recognition of unknown images with possible rotation.  The Fourier-like coefficients are used to compute the rotation invariant Gaussian-Hermite moments of orders 2,3, and 4, which are then used to classify the unknown images.

Consider a collection of images in the training set $\mathrm{Tr} = \left\{ I_1, I_2, \ldots, I_K\right\}$, where each $M\times M$ image has been transformed to a $M^2\times 1$ vector.  For each vector $I_i$ in the training set, the Fourier-like coefficient vector ${\bf c}_i$ is computed as in Equation~\eqref{eq:dantzig} where the collocation matrix $X$ is slightly changed.  Instead of using the Hermite polynomials, each entry in $X$ is an evaluation of a Hermite function $X_{j,k} = \pi_k({\bf x}_j)\exp (-\|{\bf x}_j\|_2^2/2)$.  Throughout this example, the triangular multi-index set $\mathbb{T}_N^2$ is used.  Although the pixels in each image naturally form an equally spaced collection of nodes, for computational convenience the nodes are mapped to $Z_M\times Z_M$,  where $Z_m$ are the zeros of the $M^\mathrm{th}$ Hermite univariate polynomial. 

From the Fourier-like coefficients approximated using the optimization and collocation method as in~\eqref{eq:dantzig}, one can approximate the Gaussian-Hermite moments of an image.

\begin{theorem}\label{th:moments}
Suppose $f \in L_\omega^2(\mathbb{R}^2)$ and let $\hat{\bf c}$ be a solution to~\eqref{eq:dantzig} using the orthogonal hermite function basis.  Then the elements of $\hat{\bf c}$ approximate the geometric Gaussian-Hermite moments defined by
\[ m_{\bf n} := \iint_{\mathbb{R}^2} f(x,y) \pi_{\bf n}(x,y) e^{-(x^2+y^2)/2} \; \mathrm{d}x\mathrm{d}y.\]
\end{theorem}
\begin{proof}
Using an orthonormal hermite function basis in~\eqref{eq:dantzig} yields the approximation 
\[\hat{f}(x,y) = \sum_{{\bf n}\in W} \hat{c}_{\bf n} \pi_{\bf n}(x,y)e^{-(x^2+y^2)/2}\]
 similar to~\eqref{eq:approx final}. Substituting the above result  into the formula for the geometric Gaussian-Hermite moments yields
\begin{align*}
	m_{\bf n} &\approx \iint_{\mathbb{R}^2} \hat{f}(x,y) \pi_{\bf n}(x,y) e^{-(x^2+y^2)/2} \;\mathrm{d}x\mathrm{d}y \\
	&= \iint_{\mathbb{R}^d} \sum_{\bf t} \hat{c}_{\bf t} \pi_{\bf t}(x,y) \pi_{\bf n}(x,y) e^{-(x^2+y^2)} \;\mathrm{d}x\mathrm{d}y\\
	&= \sum_{{\bf t}\in W} \hat{c}_{\bf t} \left \langle \pi_{\bf t}, \pi_{\bf n} \right\rangle_\omega,
\end{align*}
which reduces to just the element $\hat{c}_{\bf n}$ since the hermite polynomials are orthonormal relative to the weight $\omega$.
\end{proof}
	
Although Theorem~\ref{th:moments} is written in the continuous case, the discrete case follows by a straightforward substitution of summations for integrations. Using the Hermite functions instead of the Hermite polynomials is important and guarantees that the Fourier-like coefficient vector is a close approximation to the corresponding Gaussian-Hermite moment for each index in $\mathbb{R}_N^2$. 
From a small collection of these approximated Gaussian-Hermite moments, one can determine the rotation invariant Gaussian-Hermite moments $\phi_j$ of orders 2, 3, and 4.  A list of the formulas for these rotation invariant moments is given in the Appendix.

For this experiment, seven images form the training set.  These images are 8-bit grayscale $50\times 50$ images of simplified Chinese characters, which were studied in~\cite{yangdai}.  The training set is displayed in the top row of Figure~\ref{fig:example 2 training}.  Note that Images 1-3 are visually similar, as are Images 4-5 and Images 6-7.  The testing set $\mathrm{Ts}$ is formed by rotating the images in the training set and adding noise:
\[	\mathrm{Ts} = \left\{ R_\theta v + z: v\in\mathrm{Tr}, \theta \in \{0, \pi/4, \pi/2, \ldots, 7\pi/8 \right\},\] 
where the operator $R_\theta$ performs counterclockwise rotation through an angle of $\theta$, and $z$ is a collection of independent and identically distributed normal random variables with mean $0$ and standard deviation $\sigma$.   
For any vector ${\bf x}$ in the training or testing set, let $\Phi_{\bf x} \in \mathbb{R}^{11}$ be the vector of the rotation invariant Gaussian-Hermite moments of orders 2 thru 4 corresponding to the vector ${\bf x}$.

To perform classification of the rotated images, say ${\bf x}\in\mathrm{Ts}$ is classified as a rotation of image ${\bf y}\in\mathrm{Tr}$ iff
\[ \left\| \Phi_{\bf x} - \Phi_{\bf y} \right\|_1 \leq \left\| \Phi_{\bf x} - \Phi_{\bf z} \right\|_1 \quad \forall {\bf z}\in\mathrm{Tr}.\]  
\begin{table}
	\centering
	\begin{tabular}{l||r|r||r|r}
	&\multicolumn{2}{c||}{White Noise} &\multicolumn{2}{c}{Bit-flip}\\
	$\sigma$	& Identified	&Categorized	&Identified	 &Categorized\\
	\hline
	0.00	& 1.0000	& 1.0000	& 1.0000	& 1.0000 \\
	0.05	& 0.9764	& 1.0000	& 0.8032	& 0.9932 \\
	0.10	& 0.9657	& 1.0000	& 0.6950	& 0.9593 \\
	0.15	& 0.9421	& 0.9975	& 0.6107	& 0.8782 \\
	0.20	& 0.9175	& 0.9979	& 0.5007	& 0.7653 \\
	0.25	& 0.9079	& 0.9943	& 0.4000	& 0.6250 \\
	\end{tabular}
\caption{Average ratio of images correctly identified and sorted into the three categories over 50 simulations using all images in the testing data set with white Gaussian noise at level $\sigma$ and bit-flipped pixels at proportion $\sigma$.}
\label{tab:classification accuracy}
\end{table}

To test the accuracy of the described classification method, the elements in $\mathrm{Ts}$ were classified 50 times each for $\sigma = 0, 0.05, 0.10, 0.15, 0.20$, corresponding to $0\%, 5\%,$ etc.\ noise corruption.  The average percentage of correctly classified digits in the testing data set for each noise level is shown in Table~\ref{tab:classification accuracy}.
More details for the values of $\|\Phi_{R_{\pi/2} v}-\Phi_{ y}\|_1$ for all $v,y \in \mathrm{Tr}$ are shown in Table~\ref{tab:example 2 rot90}.  It is clear from the table that although many of the images are visually similar, the rotation invariant Gaussian-Hermite moments are similar only for rotations of the same images.  Of course, if the Gaussian-Hermite moments are approximated well, then this method is expected to yield 100\% classification accuracy.  The novelty comes from computing the method used to compute the Gaussian-Hermite moments.  Employing the solution of~\eqref{eq:dantzig} is an elegant way to compute the moments without using quadrature schemes or hierarchical methods~\cite{hosny, jshen}.

\begin{figure*}
	\centering
	\includegraphics{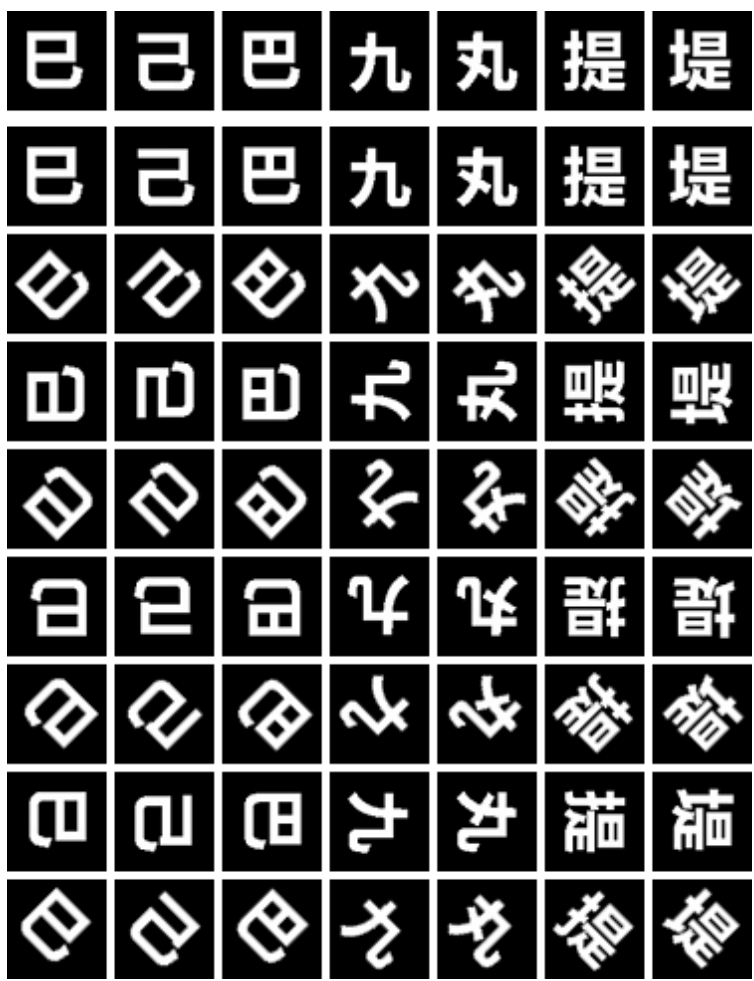}
	\caption{The images forming the training set and testing set from Example 2.  The top row contains the seven images in the training set and the remaining images form the testing set of images rotated counterclockwise $0\degree, 45\degree, 90\degree, 135\degree,$ etc.  }
	\label{fig:example 2 training}
\end{figure*}

\begin{table*}
	\centering
	\begin{tabular}{c|ccccccc}
	k\textbackslash j&1&2&3&4&5&6&7\\
\hline
1&   6.9480e-09   	&4.9948		&4.5607		&7.4822		&10.5865		&6.2996		&8.6762\\
2&   4.9948		&2.1222e-09   	&2.9397		&8.0141		&11.3291		&7.4386		&7.0308\\
3&   4.5607		&2.9397		&3.4982e-09   	&6.7547		&10.4790		&6.5229		&7.5554\\
4&   7.4822		&8.0141		&6.7547		&9.8842e-10   	&6.7450		&4.0142		&5.3556\\
5&   10.5866		&11.3291		&10.4790		&6.7450		&9.2692e-10   	&6.1498		&6.7113\\
6&   6.2996		&7.4386		&6.5229		&4.0142		&6.1498		&3.9752e-09   	&3.5455\\
7&   8.6762		&7.0308		&7.5554		&5.3556		&6.7113		&3.5455		&1.9313e-09\\
	\end{tabular}
	\caption{Values of $\displaystyle \left\| \Phi_{x} - \Phi_{I_k}\right\|_1$, where $x$ is the element $I_j\in \mathrm{Tr}$ with a $90\degree$ rotation. }
	\label{tab:example 2 rot90}
\end{table*}

\section{Summary}
In this paper, a method for computing a sparse generalized Fourier series with orthogonal polynomial bases of multivariate functions was presented.  The method first uses a truncated set of multi-indices, then approximates the coefficients using the solution of an optimization problem involving a collocation model.
Several examples were presented to illustrate the accuracy and utility of the proposed method.  In the first set of experiments, the sparse generalized Fourier series approximation of several test functions were computed, and the approximation errors were presented.  In the second set of experiments, the coefficients in the sparse generalized Fourier series were computed and used as rotation invariant feature vectors to perform pattern recognition on unknown rotated images.

\bibliographystyle{siam}

\newpage
\appendix
Below are the complete and independent set of rotation invariants of Gaussian-Hermite moments of orders 2, 3 and 4~\cite{yangprl}.  
\begin{align*}
	\phi_1 = &m_{20} + m_{02}\\
	\phi_2 = &(m_{30}+m_{12})^2 + (m_{03}+m_{21})^2\\
	\phi_3 = &(m_{20}-m_{02})[(m_{30}+m_{12})^2 - (m_{03}+m_{21})^2] \\
	&+4m_{11}(m_{30}+m_{12})(m_{03}+m_{21})\\
	\phi_4 = &m_{11}[(m_{30}+m_{12})^2 - (m_{03}+m_{21})^2]\\
	&-(m_{20}-m_{02})(m_{30}+m_{12})(m_{03}+m_{21})\\
	\phi_5 = &(m_{30}-3m_{12})(m_{30}+m_{12})\times\\
		&\times[(m_{30}+m_{12})^2 - 3(m_{03}+m_{21})^2]\\
	&+(m_{03}-3m_{21})(m_{03}+m_{21})\times\\
	&\times[(m_{03}+m_{21})^2 - 3(m_{30}+m_{12})^2]\\
	\phi_6 = &(m_{30}-3m_{12})(m_{03}+m_{21})\times\\
	&\times[(m_{03}+m_{21})^2-3(m_{30}+m_{12})^2]\\
	&-(3m_{21}-m_{03})(m_{30}+m_{12})\times\\
	&\times[(m_{30}+m_{12})^2 - 3(m_{03}+m_{21}^2]
\end{align*}
\begin{align*}
	\phi_7 = &m_{40}+2m_{22}+m_{04}\\
	\phi_8 = &(m_{40}-m_{04})[(m_{30}+m_{12})^2 - (m_{21}+m_{03})^2]\\
	&+4(m_{31}+m_{13})(m_{30}+m_{12})(m_{21}+m_{03})\\
	\phi_9 = &(m_{31}+m_{13})[(m_{30}+m_{12})^2 - (m_{21}+m_{03})^2]\\
	\phi_{10} = &(m_{40}-6m_{22}+m_{04})[(m_{30}+m_{12})^4\\
	&-6(m_{30}+m_{12})^2(m_{21}+m_{03}^2 + (m_{21}+m_{03})^4]\\
	&+16(m_{31}-m_{13})(m_{30}+m_{12})(m_{21}+m_{03})\times\\
	&\times[(m_{30}+m_{12})^2-(m_{21}+m_{03})^2] \\
	\phi_{11} = &(m_{40}-6m_{22}+m_{04})(m_{30}+m_{12})(m_{21}+m_{03})\times\\
	&\times[(m_{21}+m_{03})^2-(m_{30}+m_{12})^2]\\
	&-(m_{31}-m_{13})[(m_{30}+m_{12})^4\\
	&-6(m_{30}+m_{12})^2(m_{21}+m_{03})^2 + (m_{21}+m_{03})^4].
\end{align*}
\end{document}